\documentclass[reqno,12pt]{amsart}
\usepackage{amssymb}
\usepackage{amsmath}
\usepackage{enumerate}
\usepackage{amsfonts}
\usepackage{color}
\usepackage{graphics}
\usepackage{comment}

\usepackage[small]{pb-diagram}

\newtheorem{theorem}{Theorem}[section]
\newtheorem{lemma}[theorem]{Lemma}

\theoremstyle{remark}

\theoremstyle{remark}
\newtheorem*{note}{Remark}

\theoremstyle{remark}
\newtheorem*{notation}{\textbf{Notation}}

\theoremstyle{definition}
\newtheorem{definition}[theorem]{Definition}
\newtheorem{example}{Example}

\numberwithin{equation}{section}

\DeclareMathOperator{\supp}{supp}

\DeclareMathOperator{\sign}{sign}

\newcommand{\e}{\text{\bf E}}
\newcommand{\p}{\text{\bf P}}

\renewcommand{\sp}{\par\vspace{1mm}}
\renewcommand{\L}{\text{L\'{e}vy} }
\newcommand{\F}{\ensuremath{\mathcal F_{\mathrm{cconv}}(\R^d)} } 
\newcommand{\FK}{\ensuremath{\mathcal F_{\mathrm{cconv}}^K(\R^d)} }
\newcommand{\K}{\ensuremath{\K_{\text{conv}}(\R^d)} }

\newcommand{\N}{\ensuremath{\mathbb{N}}}

\newcommand{\R}{\ensuremath{\mathbb{R}}}

\newcommand{\s}{\ensuremath{\mathbb{S}}}
\newcommand{\B}{\ensuremath{\mathbb{B}}}

\newcommand{\one}{\ensuremath{\mathbf 1}}

\def\8{\infty}
\def\blab#1{\begin{equation}\label{#1}}
\def\elab{\end{equation}}

\def\blab#1{\begin{equation}\label{#1}}
\def\elab{\end{equation}}

\usepackage[toc,page]{appendix}

\title[Fuzzy L\'{e}vy subordinators]{L\'{e}vy subordinators in cones of fuzzy vectors}

\author[J. Schneider]{Jan Schneider}
\address{Department of Computer Science and Management\\
Wroclaw University of Science and Technology\\
ul. Smoluchowskiego 25\\
50-372 Wroclaw, Poland}
\email{jan.schneider@pwr.edu.pl}

\author[R. Urban]{Roman Urban}
\address{Institute of Mathematics\\
Wroclaw University\\
Plac Grunwaldzki 2/4\\
50-384 Wroclaw, Poland}
\email{urban@math.uni.wroc.pl}

\subjclass[2000]{}
\keywords{L\'evy processes in Banach spaces, $K$-positive fuzzy L\'{e}vy process, cone-subordinators, cones in Banach spaces, Pettis integral, Bochner integral, convex sets, fuzzy sets, fuzzy vectors, support function, Hausdorff distance}

\begin{document}

\begin{abstract}
The general problem of how to construct stochastic processes which are confined to stay in a predefined cone (in the one-dimensional but also multi-dimensional case also referred to as  \emph{subordinators}) is of course known to be of great importance in the theory and a myriad of applications.\par
But fuzzy stochastic processes are considered in this context for the first time in this paper:\par
By first relating with each proper convex cone $C$ in $\R^{n}$ a certain cone of fuzzy vectors $C^*$ and subsequently using some specific Banach space techniques we have been able to produce as many pairs $(L^*_t, C^*)$ of fuzzy \L processes $L^*_t$ and cones $C^*$ of fuzzy vectors such that $L^*_t$ are $C^*$-$\,$subordinators.
\end{abstract}

\maketitle
\section*{Summary \& Structure}

In this note we show how a cone-valued \L subordinator in the space of $d$-dimensional fuzzy vectors may be obtained.
\par
\subsection*{Technique:}\sp
The starting point of our considerations is, as so often, the one-to-one isometric correspondence between the metric space of fuzzy vectors, $\F$, and the associated convex cone $\mathcal{H}=j(\F)$ of support functions contained in the Banach space $L^{p}\bigl((0,1]\times \s^{d-1}\bigr),\, p\in[1,\infty]$, defined by the (canonical) embedding $j:$
$$j(x^*) = s_{x^*} (\cdot,\cdot). $$
We then take the subset $\FK$ of $\F$ consisting of only those fuzzy vectors $x^*\in\F$ whose characterizing functions $\xi_{x^*}$ have their supports contained in a fixed, proper convex cone: $\supp{\xi_{x^*}} \subseteq K\in \R^d.$\par\noindent
A stochastic process consisting of a collection of fuzzy random variables taking values in $\mathcal{F}^K_{\text{cconv}}(\R^d)$ is accordingly termed a ``$K$-positive \L process" (subordinator with respect to $\mathcal{F}^K_{\text{cconv}}(\R^d)$). \sp
Its construction is the very aim of this paper:\sp

We show that, with $K$ being a proper convex cone in $\R^d$, the induced cone of fuzzy vectors $\mathcal{F}^K_{\text{cconv}}(\R^d)$ is proper and convex and subsequently that the $j$-induced cone of support functions $\mathcal K=j\left(\FK\right)$ is also proper and convex. This property and others we prove, show the framework of the 2006 work of V. P\'{e}rez-Abreu and A. Rocha-Arteaga \cite{P-AR-A06}, on L\'evy processes in proper cones of Banach spaces, to be adaptable to our fuzzy setting.

With a little algebra the sought after $K$-positive process in \FK is obtained as the structure-preserving $j$-preimage of a $\mathcal{K}$-valued process constructed in the Banach space $L^p\left((0,1]\times \s^{d-1}\right).$

\subsection*{Motivation:}
\subsubsection*{Why do we consider L\'evy processes in cones?}
First of all the study of L\'evy processes and more general infinitely divisible, and stable distributions in cones is a natural, important topic which has been attracting the attention of many researchers working in probability theory. See for example works of P\'{e}rez-Abreu and Rocha-Arteaga \cite{P-AR-A06} on L\'evy processes with zero Brownian component in a cone of a Banach space, P\'{e}rez-Abreu and Rosi\'{n}ski \cite{PerRos07} on infinitely divisible distributions in cones or the work of Davydov, Molchanov, and Zuyev \cite{DavMolZuy08} on stable distributions on cones.

\subsubsection*{Why do we consider L\'evy processes with zero Brownian component?}
For a fuzzy L\'evy process to incorporate a Brownian motion component it would be necessary to first be able to construct a Brownian motion contained in a cone in an infinite dimensional Banach space.
The general Levy triplet could then be constructed.
This however is a nontrivial and important problem, which many good researchers have been and are currently working on. Partial results in this and related directions are due to, to name a few: Banuelos, DeBlassie, Garbit, Raschel, Smits \cite{BanSmi97,BanDeb06a,BanDeb06b,Deb12,GarRas16} who investigate processes which are constrained to stay in a given cone for some finite time $t \leq T.$
See also the monograph of Da Pratto and Zabczyk \cite{DapZab02}.
To our knowledge, to this date no such process has been constructed.

In view of this knowledge we had to concentrate our findings on L\'evy processes whose generator does not contain a Brownian component.

In fact our original motivation for this present work stems from our interest in, and research done on, fuzzy Brownian Motion and fuzzy Central Limit Theorems~\cite{SU17}.\sp

Indeed, up to this point, of all fuzzy \L processes really only fuzzy Brownian motion $B_t^*$ has been the center of attention of, and investigated by a number of authors (and by several different methods) but then in any case the achieved construction has been of the form
$$B^*_t(\omega)=\e B_t^*\oplus\one_{\{b_t(\omega)\}},$$
where $b_t$ is the standard Brownian motion in $\R^d,$ $\e$ stands for the (Bochner, Aumann or Frechet) expectation operator, and $\one_{\{b_t\}}$ stands for the characterizing function of the crisp vector $b_t(\omega)$ (see \cite{Bon12}). Thus, for a fixed $\omega,$ $B_t^*(\omega)$ is a fuzzy vector $m^*=\e B_t^*$ (which does not depend on $t$) moving along the trajectory of a standard (non-fuzzy) Brownian motion $b_t(\omega).$\sp

This very substantial limitation is due largely to difficulties related to the $j$-embedding being homogeneous positively only: $j(x^*\oplus (-1)\cdot y^*) \neq j(x^*) - j(y^*),$ and the resulting lack of closedness of the cone $\mathcal{H}$ with respect to the operation of subtraction in this sense.\sp

This present construction resolves this issue of closedness, but by its very design produces a L\'{e}vy process whose defining triplet has a zero Brownian motion component, that is we receive a process formed by the sum of a compound Poisson process and a square integrable pure jump martingale.\sp

\subsection*{Future perspective:}

This paper gives an important result for \L processes  with zero Brownian component.\sp

Regarding the definition of fuzzy Brownian motion we have recently been developing a very different and novel approach, namely via Dirichlet forms, and are expecting to publish our results shortly.\sp

Another possible agenda for future research may include applications of the obtained theoretical result in physics and finance, or wherever else \L subordinators play a role.

\subsection*{Structure of the paper:}
The text is divided into the following sections:\sp

In section 1 basic definitions of fuzzy vectors, variables and processes as used in this paper are gathered.\par
Section 2 gives the definitions of general Banach space valued \L processes and subordinators. Two theorems relevant for this paper are stated.\par
Section 3 is a review of how Banach spaces and fuzzy vectors are usually linked by support functions techniques.\par
Section 4 gives the actual construction.\sp

\noindent The appendices contain definitions of the Gelfand-Pettis and Bochner integrals as well as the sketch of a proof of Theorem~\ref{Th10}~\cite{P-AR-A06}.

\section{Preliminaries on fuzzy vectors, variables and processes}

\subsection{\textsc{Fuzzy vectors}}
A {\em $d$-dimensional fuzzy vector $x^*$} is defined by and may be identified with its (real-valued) {\em characterizing function} $\xi_{x^*}(\cdot):\R^d\to [0,1],$ (see e.g. \cite{Vie11}). In this paper we work with vectors whose characterizing functions satisfy
\begin{definition}\label{defxi}
The characterizing function $\xi_{x^*}$
of a $d$-dimensional fuzzy vector $x^*$ is a function $\xi_{x^*}:\R^d\to\R$ satisfying:
\begin{itemize}
\item[1)] $\xi_{x^*}:\R^d\to [0,1],$
\item[2)] $\supp{\xi_{x^*}}$ is bounded,
\item[3)] for every $\alpha\in(0,1]$ the so called \emph{$\alpha$-cut $C_\alpha(x^*)$} of the fuzzy vector $x^*$,
$C_\alpha(x^*)= \{x\in\R^d:\xi_{x^*}(x)\geq\alpha\}$
is a non-empty, compact, and convex set, as is $C_0(x^*):=\overline{\{x\in\R^d: \xi_{x^*}(x)>0\}}=\supp\xi_{x^*}.$
\end{itemize}
\end{definition}
\begin{notation}
In the following we denote the set of all $d$-dimensional fuzzy vectors satisfying Definition \ref{defxi} by $\mathcal F_{\text{cconv}}(\R^d)$.
\end{notation}
\subsubsection{Fuzzy arithmetic}
We start with the Minkowski arithmetic performed on subsets of $\R^d.$
\begin{definition}
Let $A,B\subset\R^d$ and $\lambda\in\R$. Then
\begin{equation*}
\begin{aligned}
A+B:=&\{a+b:\,a\in A,\,B\in B\},\\
A\cdot B:=&\{ab:\,a\in A,\,b\in B\},\\
\lambda\cdot A:=&\{\lambda a:\,a\in A\}.
\end{aligned}
\end{equation*}
\end{definition}
The Minkowski arithmetic of sets directly induces an arithmetic of $d$-di\-men\-sio\-nal vectors from $\mathcal F_{\text{cconv}}(\R^d)$ via $\alpha$-cuts:
\begin{definition}
The sum $x^*\oplus y^*$ and multiplication $x^*\odot y^*,$ of two fuzzy $d$-dimensional fuzzy vectors $x^*$ and $y^*$
are defined via $\alpha$-cuts as follows
\begin{equation*}
\begin{aligned}
C_\alpha(x^*\oplus y^*)=\; &C_\alpha(x^*)+C_\alpha(y^*),\\
C_\alpha(x^*\odot y^*)=\; &C_\alpha(x^*)\cdot C_\alpha(y^*),\\
\end{aligned}
\end{equation*}
Similarly, the multiplication of a fuzzy vector $x^*$ by a real number $\lambda,$ $\lambda\odot x^*$ is defined by the equation
\begin{equation*}
C_\alpha(\lambda\odot x^*)=\lambda\cdot C_\alpha(x^*).
\end{equation*}
\begin{notation}
Let $C(\R^d)$  denote the set of all non-empty, closed subsets of $\R^d.$ By $C_{\text{c}}(\R^d)$ ($C_{\text{conv}}(\R^d)$, resp.) we denote the non-empty space of all compact subsets of $\R^d$ (the non-empty space of all closed convex subsets of $\R^d$, resp.) and finally, $C_{\text{cconv}}(\R^d)$ is the space of all non-empty compact and convex subsets of $\R^d$.\vspace{1mm}\par\noindent
$\R^d$ is equipped with the classical $\ell^2$-norm $\|x\|_{\ell^2}=\left(\sum_1^dx_i^2\right)^{1\slash 2}$ and the inner product $\langle x,y\rangle =\sum_1^dx_iy_i.$
\end{notation}

\smallskip
We assume throughout the paper that $d>1.$ In the case $d=1$ we would be dealing with fuzzy intervals (numbers), which require a different set of techniques.
\subsection{Support functions}\label{support}
We start with the definition of the support function of a closed convex set in $\R^d:$\sp
Let $C\subset C_{\text{cconv}}(\R^d).$ By  $\s^{d-1}$ we denote the unit sphere in $\R^d,$ i.e.,  $\s^{d-1}=\{x\in\R^d: \|x\|_{\ell^2}=1\}.$
\begin{definition} The {\em support function of a set} $C\subset C_{\text{cconv}}(\R^d)$ is the function  $s_C:\s^{d-1}\to\R$ defined by
\begin{equation*}
s_C(u)=\sup_{a\in C} \langle u,a\rangle,\qquad u\in \s^{d-1}.
\end{equation*}
\end{definition}
Now let $x^*$ be a $d$-dimensional fuzzy vector. By 3) of Definition~\ref{defxi} its $\alpha$-cuts $C_\alpha(x^*)$ belong to $C_{\text{cconv}}(\R^d)$ so we can define the {\em support function $s_{x^*}$ of a fuzzy vector} $x^*$ as follows:
\begin{definition}\label{sfff}
\begin{equation}\label{sf}
s_{x^*}(\alpha,u)=\sup_{a\in C_\alpha(x^*)}\langle u,a\rangle,\qquad\alpha\in(0,1]\text{ and }u\in \s^{d-1}.
\end{equation}
\end{definition}
The support function $s_{x^*}(\cdot,\cdot)$ of $x^*\in\mathcal F_{\text{cconv}}(\R^d)$ has the following properties (\cite{DK94,Kor97}):
\begin{itemize}
\item [(i)] For every $\alpha\in(0,1]$, $s_{x^*}(\alpha,\cdot):\s^{d-1}\to\R$ is a  continuous function.
\item [(ii)] The support function is positively homogeneous with respect to the $u$ variable, i.e., for all real $\lambda>0$ and all $\alpha\in(0,1]$, $s_{x^*}(\alpha,\lambda u)=\lambda     s_{x^*}(\alpha,u).$
\item [(iii)] For every $\alpha\in(0,1]$, $s_{x^*}(\alpha,\cdot)$ is sub-additive, i.e., for all $u,v\in \s^{d-1},$ $s_{x^*}(\alpha,u+v)\leq  s_{x^*}(\alpha,u)+s_{x^*}(\alpha,v).$
\item [(iv)] For every $u\in \s^{d-1}$ the function $s_{x^*}(\cdot,u):(0,1]\to\R$ is left continuous and non-increasing, i.e., for all $0<\alpha\leq\beta\leq 1,$ $s_{x^*}(\alpha,u)\geq s_{x^*}(\beta,u).$
\end{itemize}

The proof of the following lemma is straightforward:
\begin{lemma}\label{addsup}
For every $x^*, y^*\in\mathcal F_{\text{cconv}}(\R^d)$ and for every $\lambda\in\R,$ we have that
\begin{equation}\label{additive}
s_{x^*\oplus y^*}(\alpha,u)=s_{x^*}(\alpha,u)+s_{y^*}(\alpha,u)
\end{equation}
and
\begin{equation}\label{pseudolin}
s_{\lambda\odot x^*}(\alpha,u)=|\lambda| s_{\sign{(\lambda)} \odot x^*}(\alpha,u).
\end{equation}
\end{lemma}
\end{definition}\sp

\subsection{\textsc{Fuzzy Random variables and processes}}
We use the following definition of a fuzzy random variable:

\begin{definition}
By a fuzzy random variable we understand a measurable function $X^*:(\Omega,\mathcal E,\p)\to (\mathcal F_{\text{cconv}}(\R^d),\mathcal B) =:S.$\sp
\noindent Here $(\Omega,\mathcal E,\p)$ is some probability space and $S$ is the space of fuzzy vectors, equipped with a metric $d_p$, for some $p\in[1,\infty].$ The associated Borel $\sigma$-field $\mathcal B$ is generated by the balls which are open in the chosen metric $d_p$.\sp

As usual, $X^*$ is termed an $(\mathcal E - \mathcal B)$-measurable function if and only if for every $B\in\mathcal B$ the inverse image ${X^*}^{-1}(B):=\{\omega\in\Omega:X^*(\omega)\in B\}$ belongs to the $\sigma$-field $\mathcal E.$
\end{definition}

\noindent
The notion of independence of fuzzy random variables transfers from the classical case verbatim:
\begin{definition}
Two fuzzy random variables $X^*$ and $Y^*$ are independent if and only if $\p(X^*\in B_1\text{ and }Y^*\in B_2)=\p(X^*\in B_1)\cdot\p(Y^*\in B_2),$ for all $B_1,B_2$ belonging to the pertinent $\sigma$-field $\mathcal{B}.$
\end{definition}

\begin{definition}
A fuzzy random process is a collection of fuzzy random variables indexed by some set $T.$
\end{definition}

\begin{definition}
The support function $s_{X^*}$ of a fuzzy random variable $X^*$ is defined by the equation
\begin{equation*}
s_{X^*}(\alpha,u)(\omega)=s_{X^*(\omega)}(\alpha,u)
\end{equation*}
\end{definition}

\begin{notation} To notationally distinguish real objects from their fuzzy counterparts we use the starred form for the fuzzy object, i.e. $x\in\R^d$, $x^*\in\F$, for vectors; $X$, $X^*$ for random variables, and $X_t$, $X_t^*$ for random processes.
\end{notation}

\noindent For general reference regarding general and probabilistic fuzzy concepts see the monograph~\cite{Vie11}.

\section{Preliminaries on L\'evy processes and subordinators in Banach spaces}
\subsection{\textsc{\L processes}}
Let $B$ be a real, separable Banach space (equipped with a norm $\|\cdot\|$).
\begin{definition}\label{B_process}
A $B$-valued L\'evy process $X_t,$ $t\geq 0$ is a stochastic process defined on some probability space $(\Omega,\mathcal E,\p)$ having properties:
\begin{itemize}
\item[(i)] The process $X_t$ starts from zero, i.e., $X_0=0$ a.s.,
\item[(ii)] $X_t$ has independent and stationary increments,
\item[(iii)] $X_t$ is stochastically continuous with respect to the norm, i.e., for every $\varepsilon>0$, $\lim_{s\to t}\p(\|X_t-X_s\|>\varepsilon)=0.$
\item[(iv)] The trajectories of the process $X_t$ are a.s. c\`adl\`ag, i.e., are right-continuous and have left-limits with respect to the norm.
\end{itemize}
\end{definition}
Let $L_t$ be a $B$-valued L\'evy processes.
Its L\'evy-Khinchin formula, which gives the Fourier transform of $L_t$ at a fixed time $t,$ is well known (see e.g. \cite{GS75}). For easy reference we recall the definition here:\par
Let $D_0 = \{x:0 <\|x\|\leq 1\}.$ Then, for every linear functional $\ell\in B^*$ (the topological dual space of $B$),
\begin{equation*}
\e e^{i\ell(X_t)}=e^{t(-\frac{1}{2}\langle A\ell,\ell\rangle+i\ell(\gamma)+\psi(\ell,\Pi))},
\end{equation*}
with
\begin{equation*}
\psi(\ell,\Pi)=\int\left(e^{i\ell(x)}-1-i\ell(x)\one_{D_0}(x)
\right)
\Pi(dx),
\end{equation*}
where $\gamma\in B,$ $A$ is a non-negative self-adjoint operator from $B^*$ to $B$, and $\Pi(dx)$ is the L\'evy measure on $B\setminus\{0\}$ satisfying, for every $\ell\in B^*,$
\begin{equation*}
\int_{D_0}|\ell(x)|^2\Pi(dx)<+\infty.
\end{equation*}
\par
The so called {\em generating triplet} of parameters: $(A,\Pi,\gamma)$ of the L´evy process $X_t$ is unique.
\par
The following result states that in Banach spaces there is a one-to-one correspondence between \L processes and {\em infinitely divisible probability measures.}\footnote{Recall, that a probability measure $\mu$ on a Banach space is called {\em infinitely divisible} if for every $n\in\N$ there exists a probability measure $\mu_n$ such that the $n$-fold convolution of $\mu_n$ gives $\mu,$ i.e., $\mu_n^{*n}=\mu.$}
\begin{theorem}\label{inf-div}
Let $\mu$ be an infinitely divisible probability measure in a Banach space $B.$ Then there is a $B$-valued \L process $L_t$ such that $L_1$ has the law $\mu.$ On the other hand if $L_t$ is a $B$-valued \L process then the law of $L_1$ is an infinitely divisible probability measure in $B.$
\end{theorem}
\begin{proof}
For the proof see the classical work by Gikhman and Skorohod \cite{GS75}.\vspace{1mm}

\end{proof}\sp
\noindent
For general reference regarding \L processes please see~\cite{App04} or~\cite{Sat04}. For general reference regarding functional analysis and specifically Banach spaces see~\cite{Yos71}.
\par
\subsection{\textsc{$\mathcal K$-subordinators in Banach spaces}}$ $\sp
Remember that our principal aim in this note is to construct a fuzzy L\'evy process which starts from a proper convex cone $\mathcal{K}$ in a Banach space $B$ and remains in this cone as $t\to\infty.$\sp

We start this paragraph with the general definition of a convex cone in any commutative algebraic structure $\mathcal S$ over some ordered field $F$ which allows for the concept of ``positive scalar".

\begin{definition} A subset $\mathcal K$ of $\mathcal S$ is said to be a {\em convex cone} in $\mathcal S$ if it is closed with respect to addition and to multiplication by positive scalars, i.e., if
\begin{itemize}
\item[(i)] for $x, y \in\mathcal K$ then $x+y \in \mathcal K.$
\item[(ii)] for every $0<\lambda\in F$ and for all $x\in\mathcal K$ it follows that $\lambda\cdot x\in\mathcal K.$
\end{itemize}
\end{definition}

We also use the following
\begin{definition}
A convex cone $\mathcal{K}$ is called {\em proper} if the condition that $x\in K$ and $-x\in K$ implies $x=0.$\sp

A proper cone $\mathcal{K}$ in a Banach space $B$ induces a partial order on $B.$ We write $x_1\leq_{\mathcal{K}}x_2$ iff $x_2-x_1\in \mathcal{K}.$ Accordingly a sequence $x_n$ of elements of $B$ is called {\em $\mathcal{K}$-increasing} ({\em $\mathcal{K}$-decreasing}, resp.) {\em in $B$} if, for all $n\in\N,$ $x_n\leq_{\mathcal{K}}x_{n+1}$ ($x_{n+1}\leq_{\mathcal{K}}x_n$, resp.).
\end{definition}

Recall that in the real values setting a subordinator is defined to be a \L process (often within another stochastic process), which is a.s. increasing in the usual, real sense.\sp
It is not difficult to see that a $B$-valued L\'evy process is $\mathcal K$-increasing if and only if it is $\mathcal K$-valued.
Hence the following definition of a subordinator in the setting of random processes taking values in Banach spaces suggests itself:
\begin{definition}
If a $B$-process is $\mathcal K$-valued for some proper convex cone $\mathcal K\subset B$, we say that it is a {\em $\mathcal K$-subordinator}.
\end{definition}

\subsection{Relevant theorems and definitions}
We end this section with two theorems, which are to be cornerstones of the construction we have set out to achieve. First we need another  definition:
\begin{definition}
Let $\Pi(dx)$ be the L\'evy measure on $B\setminus\{0\}.$ An element $\mathcal I_\Pi\in B$ is a {\em $\Pi$-Pettis centering} if, for every $\ell\in B^*,$ $\int_{D_0}|\ell(x)|\Pi(dx)$ is finite and $\ell(\mathcal I_\Pi)=\int_{D_0}\ell(x)\Pi(dx).$ In this case we write $\mathcal I_\Pi=\int_{D_0}xd\Pi(dx).$ (See Appendix A for a definition of the Gelfand-Pettis integral.)
\end{definition}
\begin{theorem}[P\'erez-Abreu and Rocha-Artega, {\cite[Theorem~10]{P-AR-A06}}]\label{Th10}
Let $\mathcal K$ be a proper cone in a separable Banach space $B.$ Let $X_t$, $t\geq 0$ be a L\'evy process in $B$ with generating triplet $(A,\Pi,\gamma).$ Assume the following three conditions:
\begin{itemize}
\item[(a)] $A=0,$
\item[(b)] $\Pi(B\setminus \mathcal \mathcal K)=0.$
\item[(c)] there exists a $\Pi$-Pettis centering $\mathcal I_\Pi=\int_{D_0}x\Pi(dx)$ such that $\gamma_0:=\gamma-\mathcal I_\Pi\in \mathcal K.$
\end{itemize}
Then the process $X_t$ is a $\mathcal K$-subordinator, i.e., for all $t\geq0,$ $X_t\in \mathcal K.$
\end{theorem}
As a corollary from Theorem~\ref{Th10} and the L\'evy-Khinchin representation for \L processes of bounded variation derived in \cite{GS75} the following result is obtained.
\begin{theorem}[P\'erez-Abreu and Rocha-Artega, {\cite[Corollary~13]{P-AR-A06}}]\label{Levy_bounded_variation}
Let $\mathcal K$ be a proper cone of $B.$ Let $\{Z_t : t\geq 0\}$ be a $B$-valued \L process with generating triplet $(A,\nu,\gamma)$ satisfying (a), (b), (c) of Theorem~\ref{Th10} as well as the additional condition $$\int_{D_0}\|x\|\nu(dx)<\infty.$$
Then $\int_{D_0}x\nu(dx)$ is a Bochner integral and the process $Z_t$ is a subordinator of bounded variation.
\end{theorem}
Some basic facts concerning the Gelfand-Pettis integral as well as the Bochner integral are included in Appendix A.\sp

We will also need the following
\begin{definition}
Let $\mathcal K$ be a proper cone in a separable Banach space $B.$ A $\mathcal K$-subordinator is said to be {\em a regular subordinator in $\mathcal K$} or {\em $\mathcal K$-valued regular subordinator} $Z_t$ if its L\'evy-Khinchin representation has the special form:
\begin{equation}\label{specialform_of_L_K}
e^{i\ell(Z_t)}=
\exp\left(t\int_\mathcal K\left(e^{i\ell(x)}-1\right)\nu(dx)
+it\ell(\gamma_0)\right),
\end{equation}
for all $\ell\in B^*.$
\end{definition}
\begin{note}Observe that (a), (b) and (c) of Theorem~\ref{Th10} imply that the L\'evy-Khinchin representation is of the special form given by \ref{specialform_of_L_K}.
\end{note}
\section{The j-embedding. Linking Banach spaces and fuzzy vectors}
This standard knowledge section shows how \F may be isometrically embedded into a certain Banach space. This will in turn, by subsequent re-inversion, allow us to apply Theorem~\ref{Th10} in the fuzzy setting of the space \F.\sp

\subsection{Metrics on $\F$}
To begin with let us introduce appropriate metrics on the set $\mathcal F_{\text{cconv}}(\R^d).$  First recall the Hausdorff metric on $C_{\text{cconv}}(\R^d)$:
\begin{definition}
The {\em Hausdorff metric} $d_H$ on the space  $C_{\text{cconv}}(\R^d)$ is given by the following  formula:
\begin{equation*}
d_H(A,B)=\max\left\{\sup_{a\in A}\inf_{b\in B}\|a-b\|_{\ell^2},\;\sup_{b\in B}\inf_{a\in A}\|a-b\|_{\ell^2}\right\},
\end{equation*}
where $A,B\in C_{\text{cconv}}(\R^d)$.
\end{definition}
Based on the Hausdorff metric $d_H$ one can construct $L^p$-metrics, $1\leq p\leq\infty,$ on $\mathcal F_{\text{cconv}}(\R^d)$ via $\alpha$-cuts:\par\vspace{1mm}
For $1\leq p<\infty:$
\begin{equation}\label{dp}
d_p(x^*,y^*)=
\left(\int_0^1d_H\bigl(C_\alpha(x^*),C_\alpha(y^*)\bigr)^p
d\alpha\right)^{1\slash p},
\end{equation}
\indent and for $p=\infty$:
\begin{equation}\label{d8}
d_\infty(x^*,y^*)=\sup_{\alpha\in [0,1]}d_H\bigl(C_\alpha(x^*),C_\alpha(y^*)\bigr),\;x^*,
y^*\in\mathcal F_{\text{cconv}}(\R^d).
\end{equation}

The space $\mathcal F_{\text{cconv}}(\R^d)$ equipped with $d_\infty$ is a complete metric space (but not separable). For $1\leq p<\infty$ the metric space $(\mathcal F_{\text{cconv}}(\R^d),\,d_p)$ is separable but not complete (see \cite{DK90,DK99,Kra02}). In this place we have to emphasise that the lack of completeness is not a problem in our paper as we do not need it in our reasoning.
\begin{note}
If for some other reasons one would need a complete metric space then one can use a completion, see \cite{Kra02,Fen01}.
\end{note}

Now we will define the Banach space $B$ which \F is to be embedded into:
\subsection{The Banach space $B = L^p\bigl((0,1]\times\s^{d-1},\,dxd\lambda\bigr)$}$ $\par

We are looking at the space of real valued, $p$-integrable functions on the product of (0,1] and the $n$-dimensional unit sphere $\s^{d-1}$ (where $\lambda$ is the normalized Lebesgue measure on $\s^{d-1}$):\sp

The space $B = L^p\bigl((0,1]\times\s^{d-1},\,dxd\lambda\bigr)$, (in short $L^p\left((0,1]\times\s^{d-1}\right)$) is equipped with a standard $L^p$-norm (and metric). Namely,
for $1\leq p<\infty$ the norm is given by:
\begin{equation*}
\|f\|_p=\left(\int_0^1\int_{\s^{d-1}}|f(x,u)|^pdx d\lambda(u)\right)^{1\slash p},
\end{equation*}
and the corresponding distance function is
\begin{equation}\label{rhop}
\rho_p(f,g)=\|f-g\|_p \,.
\end{equation}
For $p=\infty,$
$$\|f\|_\infty=\sup_{\alpha\in(0,1]}\sup_{u\in \s^{d-1}}|
f(\alpha,u)|\,,$$
and the corresponding metric is
\begin{equation*}
\rho_\infty(f,g)=\|f-g\|_\infty.
\end{equation*}\sp

The embedding of \F into $B$ is defined as follows:
\begin{theorem}\label{embthm}
Let $1\leq p\leq+\infty$ and
\begin{equation*}
j:\mathcal F_{\mathrm{cconv}}(\R^d)\to L^p\left((0,1]\times \s^{d-1}\right)
\end{equation*}
be defined by
\begin{equation}\label{j}
j(x^*)=s_{x^*}(\cdot,\cdot).
\end{equation}

Then the mapping $j$ is {\em positive linear}\footnote{Positive linearity of $j$ \eqref{al} follows from Lemma~\ref{addsup}.}, i.e., for all non-negative real numbers $\lambda_1,\lambda_2$ we have
\begin{equation}\label{al}
j(\lambda_1\odot x^*\oplus \lambda_2\odot y^*)=\lambda_1 j(x^*)+\lambda_2 j(y^*),\text{ for }\lambda_1,\lambda_2\geq 0.
\end{equation}\noindent

Furthermore the $j$-map is one-to-one onto its image $\mathcal{H} = j\left(\mathcal F_{\mathrm{cconv}}(\R^d)\right)$ which is a closed and convex cone in $L^p\left((0,1]\times \s^{d-1}\right).$
Moreover, for all $1\leq p\leq\infty,$ the mapping $j$ is an isometry,
\begin{equation}\label{izom}
d_p(x^*,y^*)=\rho_p(j(x^*),j(y^*)).
\end{equation}
\end{theorem}
\begin{proof}
For the proof see~\cite[p. 158]{Vie11} and the literature cited therein.
\end{proof}

\section{Construction of a $K$-positive fuzzy L\'evy process}
\begin{note}
Because we need separability of the metric space $(\F,d_p)$ we will henceforth assume that $p\in[1,\infty).$
\end{note}
The following lemma is immediate:

\begin{lemma}\label{NP}
The cone $\mathcal H=j\left(\mathcal F_{\mathrm{cconv}}(\R^d)\right)$ is not proper.
\end{lemma}
\begin{proof}
Take real vectors $x\not=0$ and $-x$ in $\R^d$. Both of them can be considered $d$-dimensional fuzzy vectors with
characterizing functions $\delta_x$ and $\delta_{-x}.$
Clearly,
\begin{align*}
j(x)=&s_x(\alpha,u)=\langle x,u\rangle \intertext{and}
j(-x)=&s_{-x}(\alpha,u)=-\langle x,u\rangle.
\end{align*}
Thus $\langle x,u\rangle$ and $-\langle x,u\rangle$ belong to $\mathcal H$ but $\langle x,u\rangle\not=0$ for $x\not=0.$
\end{proof}

It is clear from the preceding Lemma, that the subordinator we are looking to construct cannot possibly be based in $\mathcal{H}.$\par
Instead let $K\subset\R^d$ be any fixed proper convex cone. By $\FK$ we denote the set of all $d$-dimensional fuzzy vectors such that the supports of their characterizing functions are contained in $K,$ i.e.,
\begin{equation*}
\mathcal F_{\text{cconv}}^K(\R^d)=\{x^*\in\mathcal F_{\text{cconv}}(\R^d):\supp\xi_{x^*}\subset K\}.
\end{equation*}
\begin{lemma}\label{FK}
Let $K$ be a proper convex cone in $\R^d$. Then $\FK$ is a proper convex cone in $\F$
\end{lemma}
\begin{proof}
To see that $\FK$ is a convex cone is straightforward. Thus we only need to show that it is proper.

Clearly, if $x^*$ and $y^*$ are elements of $\FK$ then $\supp\xi_{x^*\oplus y^*}\subset K.$ Similarly, if $\lambda\geq 0$ and $x^*\in\FK$, then $\supp\xi_{\lambda\otimes x^*}\subset K.$ To show that $\FK$ is proper suppose that $x^*$ and $-x^*:=(-1)\odot x^*$ belong to $\FK.$ Then $\supp\xi_{x^*}$ and $\supp\xi_{-x^*}=-\supp\xi_{x^*}$ are subsets of $K$ which is a proper cone. Thus $\supp\xi_{x^*}=\{0\}.$ This implies that $x^*=0$.
\end{proof}
\begin{note}
Note that the proper convex cone \FK induces a partial order  $\leq_{\mathcal K}$ on $\mathcal F_{\text{cconv}}(\R^d)$ by defining
\begin{equation}
\begin{split}
{x^*} \leq_{\mathcal K} {y^*}\, &:\Leftrightarrow \, j^{-1}\bigl(j({y^*}) - j({x^*})\bigr) \in\FK\\
&\Leftrightarrow j({y^*}) - j({x^*})\in\mathcal K\\
&\Leftrightarrow j(x^*)\leq_\mathcal Kj(y^*),
\end{split}
\end{equation}
and elements $x^*$ belonging to the cone $\FK$ are called \emph{K-positive}, because naturally $0 \leq_{\mathcal K}x^*.$ - Every element $x^*$ which belongs to $\FK$ satisfies $\supp \xi_{x^*}\subset K.$ \par This is in full analogy to the $1$-dimensional case, where very often a fuzzy number $x^*\in\mathcal F_{\text{cconv}}(\R)$ with $\supp \xi_{x^*}\subset(0,+\infty)$ is called positive. Here the cone $K$ would be the positive half-line, i.e., $K=(0,+\infty).$
\end{note}
\begin{lemma}\label{calJ}
Let $\mathcal{K} = j\left(\mathcal F_{\mathrm{cconv}}^K(\R^d)\right).$ Then $\mathcal K\subset\mathcal H = j\bigl(\F\bigr)$ and $\mathcal{K}$ is a proper convex cone in the Banach space $L^p((0,1]\times\s^{d-1}).$
\end{lemma}
\begin{proof}
By the definition \eqref{j} of the embedding $j$, $\mathcal{K}\subset \mathcal H.$ Let $f,g\in \mathcal{K}.$ Then there are vector characterizing functions $\xi_{x^*}$ and ${\xi_y^*}$ with supports in $K$ such that $j({x^*})=f$ and $j({y^*})=g.$ Then, by \eqref{al}, for every real $\lambda>0,$ $\lambda f=j(\lambda{x^*}).$ By Lemma~\ref{FK}, $\lambda{x^*}\in\FK.$ Hence $\lambda f\in \mathcal{K}.$ Similarly, by \eqref{al}, $f+g=j({x^*})+j({y^*})=j({x^*\oplus y^*}).$ By Lemma~\ref{FK}, $\xi_{x^*\oplus y^*}\in\FK.$ Thus $\mathcal{K}$ is a convex cone.
Now we prove that $\mathcal{K}$ is proper. Let $f,-f\in \mathcal{K}.$ Then there are $x^*,y^*\in\FK$ such that $j(x^*)=f$ and $j(y^*)=-f.$ Then $j(x^*\oplus y^*)=f-f=0.$ Therefore $\xi_{x^*\oplus y^*}=\xi_{\{0\}}$ and consequently $x^*\oplus y^*=0.$ By Lemma~\ref{FK}, $\FK$ is a proper cone, thus the last equation has only the trivial solution $x^*=0$ and $y^*=0.$ This implies that $f=j(x^*)=0$ and $-f=j(y^*)=0$.
\end{proof}

To sum up and make things clearer, the following diagram illustrates the interrelation of the spaces involved:

\begin{diagram}
               &        &                            &                  &  B = L^p\left((0,1]\times \s^{d-1}\right)\\
               &        &                            &                  &     \cup                                  \\
  \mathbb{R}^d &  \rTo  &  \mathcal{F}_{\text{cconv}}(\R^d)    &   \rTo^{\qquad \qquad j}     &  \mathcal{H}                                       \\
   \cup        &        &       \cup                 &                  &     \cup                                 \\
    $K$        &  \rTo  &  \mathcal{F}_{\text{cconv}}^K(\R^d)  &    \rTo^{\qquad \qquad j}         &  \mathcal{K}                                      \\
               &        &         \cup                &                  &      \cup                                      \\
               &        &       {X^K_t}^*       &  \lTo^{\qquad \qquad j^{-1}} &  X^{\mathcal{K}}_t   \\
               &&&&
\end{diagram}

The above diagram already incorporates the processes we are about to define.
\subsubsection{Construction of a $K$-positive fuzzy L\'evy process}
The first step is to construct a Banach space valued L\'evy process $X_t^{\mathcal{K}}$ based in the proper convex cone $\mathcal K\subset\mathcal H \subset B.$ Note that, according to Lemma~\ref{NP}, the cone $\mathcal H$ is not proper so we can not simply use it in our construction. However, by Lemma~\ref{calJ} its sub-cone $\mathcal{K}$ is proper alright and as such meets the requirements of Theorem~\ref{Th10}. Based on Theorem~\ref{Th10}, a $\mathcal{K}$-valued L\'evy process $X_t^{\mathcal{K}}$ can therefore be constructed\footnote{
In Appendix B we give a brief sketch of the of the proof of Theorem~\ref{Th10} for the reader's convenience.}.\sp

The second step is to apply $j^{-1},$ the inverse of $j,$ to the process $X_t^{\mathcal{K}}$ in order to get a fuzzy L\'evy process ${X^K_t}^* = j^{-1}(X_t^{\mathcal{K}})$ taking values in the fuzzy vector cone $\FK$.\sp

In our terminology the process ${X^K_t}^* = j^{-1}(X_t^{\mathcal{K}})$ is $K$-positive.\sp

The final step is to show that the process ${X^K_t}^*$ constructed as described satisfies properties (i)-(iv) of Definition~\ref{B_process}, or rather their natural fuzzy analogues, which allows us to call the constructed process a \L process.
\begin{note} In (ii) we introduce the concept of metric
increments of the process.
\end{note}
\begin{theorem}
The process ${X^K_t}^*$ constructed above has the following properties.
\begin{itemize}
\item[(i)] The process ${X^K_t}^*$ starts from zero, i.e., ${X^K_0}^*=0$ a.s..
\item[(ii)] ${X^K_t}^*$ has independent and stationary increments of the form $$d_p\left(X^K_t,X^K_s\right),$$
for every $t>s\geq 0$ and all $1\leq p<+\infty.$
\item[(iii)] ${X^K_t}^*$ is stochastically continuous with respect to the metric $d_p$ for every $1\leq p<\infty,$ i.e., for every $\varepsilon>0$ and every $p,$
    \begin{equation*}\lim_{s\to t}\p\left(d_p\left({X^K_t}^*,
    {X^K_s}^*\right)>\varepsilon\right)=0.
    \end{equation*}
\item[(iv)] The trajectories of the process ${X^K_t}^*$ are a.s. c\`adl\`ag, i.e., are right-continuous and have left-limits with respect to the metric $d_p$ for all $1\leq p<+\infty.$
\end{itemize}
\end{theorem}
\begin{proof}
(i) This is clear. It follows from the fact that the \L process $X_t^{\mathcal K}$ may be assumed to start from zero (see Definition~\ref{B_process}) and that $j^{-1}(0)=0.$

(ii) Let $E$ be a set from $\mathcal B,$ the $\sigma$-field of Borel subsets of $\mathcal{F}_{\text{cconv}}^K(\R^d).$ We have by \eqref{izom},
\begin{equation*}
\p\left(d_p\left({X^K_t}^*,{X^K_s}^*\right)\in E\right)=\p\left(\left\|X_t^{\mathcal K}-X_s^{\mathcal K}\right\|_p\in j(E)\right).
\end{equation*}
Thus the process ${X^K_t}^*$ has stationary increments $d_p\left({X^K_t}^*,{X^K_s}^*\right).$ Independence follows from the above equality of distributions and the fact that the independence of increments $X_t^{\mathcal K}-X_s^{\mathcal K}$ of the \L process implies (by definition) independence of the random variables $\|X_t^{\mathcal K}-X_s^{\mathcal K}\|$ and $\|X_v^{\mathcal K}-X_u^{\mathcal K}\|.$
\par
(iii) It is enough to note that since $j$ is an isometry (see \eqref{izom} of Theorem~\ref{embthm}) we have
\begin{equation*}
\begin{split}
\lim_{s\to t}\p\left(d_p\left({X^K_t}^*,{X^K_s}^*\right)>\varepsilon\right)
=&\lim_{s\to t}\p\left(\left\|j({X^K_t}^*)-
j({X^K_s})^*\right\|_p>\varepsilon\right)\\
=&\lim_{s\to t}\p\left(\left\|X^{\mathcal K}_t-
X^{\mathcal K}_s\right\|_p>\varepsilon\right).
\end{split}
\end{equation*}
The right hand side tends to $0$ a.e. as $s\to t$ by property (iii) of Definition~\ref{B_process} of a \L process.
\par
(iv) Since $j$ is an isometry (see \eqref{izom}) we have
\begin{multline*}
\lim_{s\to t+}d_p\left({X^K_t}^*(\omega),{X^K_s}^*(\omega)\right)\\=
\lim_{s\to t+}\left\|j({X^K_t}^*(\omega))-j({X^K_s}^*(\omega))\right\|_p\\
=\lim_{s\to t+}\left\|X^{\mathcal K}_t(\omega)-X^{\mathcal K}_s(\omega)\right\|_p=0
\end{multline*}
since $X^{\mathcal K}_t$ is a.s. right-continuous. Thus
a.s.
\begin{equation*}
\lim_{s\to t+}d_p\left({X^K_t}^*(\omega),{X^K_s}^*(\omega)\right)=0,
\end{equation*}
i.e., the trajectories of the process ${X^K_t}^*$ are right-continuous a.s. Since $X^{\mathcal K}_t(\omega)$ has also left-limits a.s. with respect to the $L^p$ norm therefore for every $t>0$ there exists $k_\omega\in\mathcal K$ such that a.s.
\begin{equation*}
\lim_{s\to t-}\left\|X^{\mathcal K}_s(\omega)-k_\omega\right\|_p=0.
\end{equation*}
Thus by the isometry of $j,$
\begin{multline*}
\lim_{s\to t-}\left\|X^{\mathcal K}_s(\omega)-k_\omega\right\|_p=\lim_{s\to t-}d_p\left(j^{-1}(X^{\mathcal K}_s(\omega)),j^{-1}(k_\omega)\right)\\=\lim_{s\to t-}d_p\left({X^K_t}^*(\omega),j^{-1}(k_\omega)\right)=0.
\end{multline*}
Clearly $j^{-1}(k_\omega)\in\mathcal{F}^K_{\text{cconv}}(\R^d)$ and the proof is finished.
\end{proof}
\begin{example}[{\cite[Example 23]{P-AR-A06}} adapted to our setting]
As we mentioned above (Theorem~\ref{inf-div}) one can identify each infinitely divisible law $\mu$ on a Banach space $B$ with one $B$-valued \L process (Theorem~\ref{inf-div}).\sp

Take a measure $\mu$ on the Banach space $B=L^p\left((0,1]\times\s^{d-1}\right),$ which is infinitely divisible and its Fourier transform (or the characteristic function in the language of probabilists) is of the form
\begin{equation*}
(\hat\mu,\ell)=
\exp\left(c_{\alpha}^{-1}t\int_0^\infty\int_{\partial\B}
\left(e^{ir\ell(y)}-ir\ell(y)\one_\B(ry)
+it\ell(\gamma)\right)
\dfrac{\lambda(dy)}{r^{1+\alpha}}dr \right),
\end{equation*}
where $\ell\in B^*,$ $\gamma\in B,$ the constant $c_\alpha$ depends only on $\alpha,$ and $\B$ is the unit closed ball in $B,$ i.e., $\B=\{x\in B:\|x\|\leq 1\}.$ Theorem~\ref{inf-div} guarantees the existence of an $\alpha$-stable \L process $L_t$ with $\alpha$-stable law $\mu.$

Having $L_t$ we want to first construct a $\mathcal{K}$-valued subordinator $L^{\mathcal{K}}_t$ and then from there the fuzzy process ${L^K_t}^*$  based in the fuzzy cone $\FK.$ \par
As before $\mathcal K = j(\FK)$ for some $K\in\R^d.$\sp

Let $\alpha\in(0,1).$ We assume that $\lambda$ is concentrated on $\B_{\mathcal K}=\partial\B\cap\mathcal K=\{x\in\mathcal K:\|x\|=1\}.$ Let
$$\nu(C)=c_{\alpha}^{-1}\int_0^\infty\int_{\B_{\mathcal K}}\one_{C}(ry)\lambda(dy)\dfrac{dr}{r^{1+\alpha}},$$
where $C\in B_0.$ Clearly $\nu$ is concentrated on $\mathcal K.$

Now we check that properties (a), (b) and (c) assumed in Theorem~\ref{Th10} are satisfied in this example.

Condition (a) is clearly satisfied since $(\hat\mu,\ell)$ does not have a Brownian component.

The measure $\nu$ is defined in such a way that it is supported on $\mathcal K.$ So (b) is satisfied.

Considering condition (c) first we have to check that there exists a $\nu$-Pettis centering $\mathcal I_\nu=\int_{D_0}x\nu(dx)$ such that $\gamma_0:=\gamma-\mathcal I_\nu\in\mathcal K$ and $D_0=\{x:0 <\|x\|\leq 1\}.$

Notice that we have ($\|y\|=1$):
\begin{equation*}
\int_{D_0}\|x\|\nu(dx)=c_\alpha^{-1}\int_0^1
\int_{\B_{\mathcal K}}\lambda(dy)\dfrac{dr}{r^\alpha}=c_\alpha^{-1}
\dfrac{1}{1-\alpha}\lambda(\B_\mathcal K)
\end{equation*}
is finite. Theorem~\ref{ThBochner} implies the existence of the integrals $$\int_{D_0}x\nu(dx)\text{ and }\int_{B_\mathcal K}x\lambda(dx)$$
in the sense of Bochner. Clearly, the second integral belongs to $\mathcal K.$\sp

All that remains is to solve for $\gamma\in\mathcal K,$ so that $\gamma_0=\gamma-\mathcal I_\nu,$ where $\mathcal I_\nu= c_\alpha^{-1}
\frac{1}{1-\alpha}\int_{B_\mathcal K}x\lambda(dx)$ belongs to $\mathcal K.$\sp

This gives $L^{\mathcal{K}}_t$ and then from there ${L^K_t}^* = j^{-1}(L^{\mathcal{K}}_t).$
\end{example}

\begin{appendices}
\section{The Gelfand-Pettis and Bochner integrals}
The {\em Gelfand-Pettis} integral may be defined not only for functions with values in a Banach space but also in the more general setting of topological vector space valued functions. This very short exposition is restricted to Banach spaces. More on the Pettis integral one can find in \cite{Mus02}.
\begin{definition}
Let $B$ be a Banach space. Let
$f$ be a measurable $B$-valued function defined on a measure space
$(X,\mathcal B, m).$ The {\em Gelfand-Pettis integral} of
$f$ is a vector $\mathcal I_f\in B$
such that, for $\ell\in B^*,$
$$\ell(\mathcal I_f)=\int_X\ell(f(x))m(dx).$$
\end{definition}
\begin{definition}
A function $f:X\to B$ is said to be scalarly measurable, if $\ell(f)$
is measurable for every $\ell\in B^*.$
\end{definition}
\begin{definition}
A function $f:X\to B$ is strongly $\mathcal B$-measurable if there is a sequence of simple functions strongly convergent to $f$ $m$-a.e. on $X.$
\end{definition}
In spite of the relatively easy and natural definition of the Pettis integral one of the main problem in the theory of Pettis integration is to find (easily  verifiable) conditions which guarantee its existence. Two results in this direction are gathered below.
\par
But first we need one definition.
\begin{definition} A function $f:\Omega\to B$ is said to be {\em scalarly measurable}, if $\ell(f)$ is measurable for each $\ell\in B^*.$
\end{definition}
The following theorem, which gives a sufficient condition, was proved independently by Dimitrov in 1971 and Diestel in 1974.
\begin{theorem}[Dimitrov, {\cite{Dim71} and Diestel, \cite{Die73}}]
If $B$ is a separable Banach space which does not contain any isomorphic copy of $c_0$ (the space of real null sequences)
then every strongly measurable and scalarly integrable
$B$-valued
function $f$ is Pettis integrable.
\end{theorem}
A sufficient and necessary condition for Pettis-integrability was formulated by Talagrand in 1984:
\begin{theorem}[Talagrand, {\cite{Tal84}}]
Let $B$ be an arbitrary Banach space and $f:X\to B$ be scalarly integrable. Then $f$ is Pettis-integrable with respect to the measure $m$ if and only if there is a WCG space\footnote{WCG stands for {\em weakly compactly generated}.
A WCG space is a Banach space possessing a weakly compact subset whose linear span is dense.} $V\subset B$ such that $\ell(f)=0$ a.e. for every $\ell\in V^\bot$ and $T_f:B^*\to L^1(m),$ defined by $Tf(\ell)= \ell(f),$ is weakly compact.
\end{theorem}
Because of its duality definition the Pettis integral is usually called a {\em weak integral}. Now we define the {\em Bochner integral} which is a {\em strong integral}. Here we follow the presentation given in~\cite{Yos71}.
In order to define the Bochner integral first we need to introduce some preliminaries. We start with the integration of a simple function.
\begin{definition}
Let $B$ be a Banach space. A {\em simple function} $f:X\to B$ defined on a measure space $(X,\mathcal B,m)$ is a function of the form
$$f(x)=\sum_{j=1}^nc_j\one_{E_j}(x),$$
where $E_1,\ldots,E_n\in\mathcal B$
and
$c_1,\ldots,c_n$ are elements of $B.$ Then we define the {\em $m$-integral} of $f$ over $X$ by
$$\int_Xf(x)m(dx):=\sum_{j=1}^nc_jm(E_j).$$
\end{definition}
A limiting procedure allows us to integrate more complicated functions.
\begin{definition}
A function $f$ defined on a measure space $(X,\mathcal B,m)$ with values in a Banach space $B$ is {\em Bochner $m$-integrable}, if there is a sequence of simple functions $f_n$ which strongly converges to $f$ $m$-a.e. and
$$\lim_{n\to\infty}\int_X\|f(x)-f_n(x)\|m(dx)=0.$$
Then for every set $E\in\mathcal B$ the {\em Bochner integral} of $f$ over $E$ with respect to the measure $m$ is defined by
$$\int_Ef(x)m(dx)
:=s-\lim_{n\to\infty}\int_X\one_E(x)f_n(x)m(dx).$$
\end{definition}

In 1933 Bochner published the following result:
\begin{theorem}[Bochner, {\cite{Boc33}}]\label{ThBochner}
A strongly $\mathcal B$-measurable function $f$ is Boch\-ner $m$-integrable if and only if the function $x\mapsto\|f(x)\|$ is $m$-integrable.
\end{theorem}

\section{Proof of Theorem~\ref{Th10}}

Here we present sketch of the proof of \cite[Theorem 10]{P-AR-A06} which is stated in our paper as Theorem~\ref{Th10}. In the following we use the notation from our paper. For easy reference we restate the theorem:

\begin{theorem}[P\'erez-Abreu and Rocha-Artega, {\cite[Theorem~10]{P-AR-A06}}]
Let $\mathcal{K}$ be a proper cone in a separable Banach space $B.$ Let $X_t$, $t\geq 0$ be a L\'evy process in $B$ with generating triplet $(A,\Pi,\gamma).$ Assume the following three conditions:
\begin{itemize}
\item[(a)] $A=0,$
\item[(b)] $\Pi(B\setminus \mathcal{K})=0.$
\item[(c)] there exists a $\Pi$-Pettis centering $\mathcal I_\Pi=\int_{D_0}x\Pi(dx)$ such that $\gamma_0:=\gamma-\mathcal I_\Pi\in C.$
\end{itemize}
Then the process $X_t$ is a $\mathcal{K}$-subordinator, i.e., for all $t\geq0,$ $X_t\in \mathcal{K}.$
\end{theorem}

\begin{proof}[Proof of Theorem~\ref{Th10}]
For a given set $E$ in $\mathcal B_0,$ the ring of Borel sets of the Banach space $B$ with positive distance from $0,$ define {\em the jump size process} - the process which reflects the sum of jumps of the process $X_t^{\mathcal{K}}$ which occurred up to time $t$ and took place in $E,$
\begin{equation*}
X_t^{\mathcal K}(E)=\sum_{s<t}\left(X_s^{\mathcal K}-X_{s-}^{\mathcal K}\right)\one_E\left(X_s^{\mathcal K}-X_{s-}^{\mathcal K}\right),
\end{equation*}
where $$X_{s-}^{\mathcal K}=\lim_{s\uparrow t}X_s^{\mathcal K}$$ and $\one_E(\cdot)$ is the indicator function
\begin{equation*}
\one_E(x)=\begin{cases}1&\text{ if $x\in E,$}\\
0&\text{ otherwise.}
\end{cases}
\end{equation*}
For each $\varepsilon>0$ consider the jump size process $X_t^{\mathcal K,\Delta_\varepsilon},$ where $\Delta_\varepsilon=\{x\in L^p((0,1]\times\s^{d-1}):\|x\|>\varepsilon\}\in\mathcal B_0.$

Let $E\in\mathcal B_0.$ Then, by the assumption $\nu(E)=0.$ From this it follows that $X_t^{\mathcal K,\Delta_\varepsilon}$ is $\mathcal K$-valued a.s.. In fact, if there were an $E$ with $\p(X_t^{\mathcal K,E}\not=0)>0$ then, since $\p(X_t^{\mathcal K,E}\not=0)\leq 1-e^{-t\nu(E)}=0$ we would get a contradiction.

Next we notice that since $\gamma_0\in\mathcal K$ it is enough to prove that the process $\tilde X_t^{\mathcal K}:=X_t^{\mathcal K}-t\gamma_0$ is $\mathcal K$-valued a.s..

Clearly, $\tilde X_t^{\mathcal K}$ and $X_t^{\mathcal K}$ have the same jumps. Thus is $\tilde X_t^{\mathcal K,\Delta_\varepsilon}=X_t^{\mathcal K,\Delta_\varepsilon}$ and consequently $\tilde X_t^{\mathcal K,\Delta_\varepsilon}$ is $\mathcal K$-valued.

Then the final step is to show that $\tilde X_t^{\mathcal K}-\tilde X_t^{\mathcal K,\Delta_\varepsilon}$ and $\tilde X_t^{\mathcal K,\Delta_\varepsilon}$ are independent and $\mathcal K$-valued which finishes the proof (for details see \cite{P-AR-A06}).
\end{proof}
\end{appendices}

\end{document}